\documentclass[a4paper,12pt]{amsart}
\usepackage{amsmath,amssymb,amsthm}
\usepackage{a4wide}

\newcommand{\N}{{\mathbb N}}
\newcommand{\R}{{\mathbb R}}

\newcommand{\Cnt}[1][]{{\mathcal C}^{#1}}
\newcommand{\co}[1]{{#1}^{c}}
\newcommand{\conv}{\ast}
\newcommand{\csub}{\subset\subset}
\newcommand{\eps}{\varepsilon}
\renewcommand{\implies}{\Rightarrow}
\newcommand{\supp}{\mathop{\mathrm{supp}}}
\newcommand{\test}{\mathcal D}

\newcommand{\Gen}{{\mathcal G}}

\newcommand{\ster}[1]{{{}^* \mskip-1mu #1}}

\newtheorem*{thm}{Theorem}

\theoremstyle{definition}
\newtheorem*{ex}{Example}

\begin{document}
\title[Regularity of nonlinear generalized functions]{Regularity of nonlinear generalized functions:\\a counterexample in the nonstandard setting}

\author{H.~Vernaeve}\address{Dept.\ of Mathematics, Krijgslaan 281, Ghent University, 9000 Gent, Belgium}\thanks{Supported by research Grant 1.5.138.13N of the Research Foundation Flanders FWO}

\begin{abstract}
Regularity theory in generalized function algebras of Colombeau type is largely based on the notion of $\Gen^\infty$-regularity, which reduces to $\Cnt[\infty]$-regularity when restricted to Schwartz distributions. Surprisingly, in the nonstandard version of the Colombeau algebras, this basic property of $\Gen^\infty$-regularity does not hold.
\end{abstract}
\maketitle

\section{Introduction}
Generalized function algebras are differential algebras that contain (up to isomorphism) the space of Schwartz distributions as a differential subspace, and in which the product of $\Cnt[\infty]$-functions coincides with their usual product. They have been introduced by J.~F.\ Colombeau \cite{Colombeau1984} and find their main applications in the study of nonlinear PDE (e.g.\ in General Relativity \cite{GKOS,kun-stein99,stein-vick06}) and PDE with highly singular data or coefficients (e.g.\ \cite{hoop08,GarObeSymm,keyfitz,Ob92}) for which the distributional solution concept does not make sense.

A well-developed qualitative theory of generalized solutions to PDE has emerged based on the notion of $\Gen^\infty$-regularity \cite{Ob92,Scarpa93} and the corresponding $\Gen^\infty$-microlocal regularity, aiming at describing the propagation of singularities of PDE (e.g.\ \cite{Deguchi,GarHor95,GarHorObeFIO,hop06,HV-Microlocal}). The basic property which makes $\Gen^\infty$-regularity of a generalized function a suitable concept to study its regularity is that the $\Gen^\infty$-regular distributions (viewed as elements of the generalized function algebra) are exactly the $\Cnt[\infty]$-functions.

Due to inherent similarities in the construction of the Colombeau algebras with the construction of algebras of functions in nonstandard models of analysis \cite{Rob}, also a nonstandard version of the Colombeau algebras has been constructed \cite{OT98}. This variant enjoys similar, but in some aspects nicer properties than the standard algebra. E.g., the ring of (real) scalars in the nonstandard algebra is a (totally) ordered field (and not, as in the standard algebra, a partially ordered ring with zero divisors), and a Hahn-Banach extension property for continuous linear functionals holds \cite{TVfull} (which fails in the standard algebra \cite{HV-Ideals}). More generally, the full principles of nonstandard analysis (such as the Transfer Principle) are available for representatives of generalized functions in the nonstandard algebra (whereas only a restricted version holds in the standard algebras \cite{HV-NSPrinc}). The above-mentioned basic property of $\Gen^\infty$-regularity was therefore expected to hold also in the nonstandard version. However, the proof remained elusive. The goal of this paper is to construct an explicit counterexample.

\section{Preliminaries}
Let $\Omega\subseteq\R^d$ be open. We work in the so-called special Colombeau algebra $\Gen(\Omega)=\mathcal M(\Omega)/\mathcal N(\Omega)$ \cite[\S 1.2]{GKOS}, where
\begin{align*}
\mathcal M(\Omega) &= \{(u_\eps)_\eps\in \Cnt[\infty](\Omega)^{(0,1]}:\\
&(\forall K\csub\Omega) (\forall\alpha\in\N^d) (\exists N\in\N) (\exists \eps_0\in(0,1]) (\forall \eps\in(0,\eps_0)) (\sup_{x\in K}|\partial^\alpha u_\eps(x)|\le\eps^{-N})\}\\
\mathcal N(\Omega) &= \{(u_\eps)_\eps\in \Cnt[\infty](\Omega)^{(0,1]}:\\
&(\forall K\csub\Omega) (\forall\alpha\in\N^d) (\forall m\in\N) (\exists \eps_0\in(0,1]) (\forall \eps\in(0,\eps_0)) (\sup_{x\in K}|\partial^\alpha u_\eps(x)|\le\eps^{m})\}.
\end{align*}
We denote by $[u_\eps]\in\Gen(\Omega)$ the generalized function with representative $(u_\eps)_{\eps}$.\\
$\Gen^\infty(\Omega)$ is the subalgebra of those $u\in\Gen(\Omega)$ with a representative $(u_\eps)_\eps$ satisfying
\[(\forall K\csub\Omega) (\exists N\in\N) (\forall\alpha\in\N^d) (\exists \eps_0\in(0,1]) (\forall \eps\in(0,\eps_0)) (\sup_{x\in K}|\partial^\alpha u_\eps(x)|\le\eps^{-N}).\]

\smallskip

Let $\mathcal U$ be a free ultrafilter on $(0,1]$ with $(0,\delta]\in \mathcal U$ for each $\delta>0$. The nonstandard version of the special algebra is ${}^\rho\mathcal E(\Omega)=\mathcal M(\Omega)/\mathcal N(\Omega)$ \cite{Ob88,OT98}, where
\begin{align*}
\mathcal M(\Omega) &= \{(u_\eps)_\eps\in \Cnt[\infty](\Omega)^{(0,1]}:\\
&(\forall K\csub\Omega) (\forall\alpha\in\N^d) (\exists N\in\N) (\exists S\in\mathcal U) (\forall \eps\in S) (\sup_{x\in K}|\partial^\alpha u_\eps(x)|\le\eps^{-N})\}\\
\mathcal N(\Omega) &= \{(u_\eps)_\eps\in \Cnt[\infty](\Omega)^{(0,1]}:\\
&(\forall K\csub\Omega) (\forall\alpha\in\N^d) (\forall m\in\N) (\exists S\in\mathcal U) (\forall \eps\in S) (\sup_{x\in K}|\partial^\alpha u_\eps(x)|\le\eps^{m})\}
\end{align*}
(to be precise, this is the nonstandard version of the algebra in \cite{OT98} in the case where the nonstandard model is constructed using the ultrafilter $\mathcal U$ on $(0,1]$ and where the fixed positive infinitesimal $\rho\in\ster\R$ has representative $(\eps)_{\eps\in(0,1]}$).\\
${}^\rho\mathcal E^\infty(\Omega)$ is the subalgebra of those $u\in {}^\rho\mathcal E(\Omega)$ with a representative $(u_\eps)_\eps$ satisfying
\[(\forall K\csub\Omega) (\exists N\in\N) (\forall\alpha\in\N^d) (\exists S\in\mathcal U) (\forall \eps\in S) (\sup_{x\in K}|\partial^\alpha u_\eps(x)|\le\eps^{-N}).\]

Let $\phi$ be a function in the Schwartz space $\mathcal S(\R^d)$ which satisfies the moment conditions $\int_{\R^d}\phi(x)\,dx=1$ and $\int_{\R^d}x^\alpha\phi(x)\,dx=0$ for each $\alpha\in\N^d\setminus\{0\}$, and denote $\phi_\eps(x):= \frac1{\eps^d}\phi(x/\eps)$. Then the embedding $\iota$ of the space $\test'(\Omega)$ of Schwartz distributions into the Colombeau algebra $\Gen(\Omega)$ is given, for $T\in\mathcal E'(\Omega)$, by $\iota(T)=[T\conv \phi_\eps]$ \cite[\S 1.2]{GKOS}.
\smallskip

Our counterexample will be based on a careful analysis of a  proof of the result for the standard algebra. For the sake of the readability of the counterexample, we recall the proof in an elementary version:
\begin{thm}\cite{Ob92,Scarpa93}
$\Gen^\infty(\Omega)\cap\iota(\test'(\Omega))=\Cnt[\infty](\Omega)$.
\end{thm}
\begin{proof}
1. Let $T\in\mathcal E'(\Omega)$ with $\iota(T)=[T\conv \phi_\eps]\in\Gen^{\infty}(\Omega)$. Then the Fourier transform $\widehat T\in\Cnt[\infty](\Omega)\cap \mathcal S'(\R^d)$. Because $\iota(T)\in\Gen^{\infty}_c(\Omega)$, we have
\[(\exists N\in\N) (\forall \alpha\in\N^d) (\exists \eps_0>0) (\forall \eps\le\eps_0) (\sup_{x\in\R^d}\langle x\rangle^{2d}|\partial^\alpha (T\conv\phi_\eps)(x)|\le \eps^{-N}).\]
where $\langle x\rangle^2= 1+|x|^2$. For the Fourier transform $\widehat{\iota(T)} = [\widehat{T}\cdot\widehat\phi(\eps\,\cdot)]$, we have
\[|\xi^\alpha \widehat T(\xi) \widehat\phi(\eps\xi)| = |\mathcal F(\partial^\alpha (T\conv\phi_\eps))(\xi)|\le \int_{\R^d} |\partial^\alpha(T\conv\phi_\eps)(x)|\,dx.\]
Hence
\begin{equation}\label{ss-sup}
(\exists N\in\N) (\forall \alpha\in\N^d) (\exists \eps_0>0) (\forall \eps\le \eps_0) (\forall\xi\in\R^d) (|\xi^\alpha \widehat T(\xi)\widehat\phi(\eps \xi)|\le \eps^{-N}).
\end{equation}
From this, we show that $\widehat T\in\mathcal S(\R^d)$. Since $\widehat\phi(0)=\int_{\R^d}\phi=1$, there exists $\delta>0$ such that $|\widehat\phi(\eta)|\ge 1/2$ for $|\eta|\le \delta$. For $\xi\in\R^d$ with $|\xi|\ge \delta/\eps_0$, choosing $\eps=\delta/|\xi|$ in \eqref{ss-sup} (hence $\eps\le\eps_0$) yields
\begin{equation}\label{T-hat-in-S}
(\exists N\in\N) (\forall\alpha\in\N^d) (\forall \xi\in\R^d, |\xi|\ge \delta/\eps_0) (|\xi^\alpha\widehat T(\xi)|\le 2|\xi|^{N}/\delta^N).
\end{equation}
We proceed similarly for all derivatives. Hence $\widehat T\in\mathcal S(\R^d)$, and $T\in\mathcal S(\R^d)\subseteq \Cnt[\infty](\R^d)$.
\smallskip

2. If $T\in\mathcal D'(\Omega)$, then we can use a cut-off to reduce to the case $T\in\mathcal E'(\Omega)$.
\end{proof}

\section{Counterexample in ${}^\rho\mathcal E(\Omega)$}
We give a counterexample in the nonstandard algebra ${}^\rho\mathcal E(\Omega)$ for the case $\Omega=\R$. Essentially, only the step $\eqref{ss-sup}\implies\eqref{T-hat-in-S}$ in the proof of the previous section fails in this setting, which we will exploit. We use an embedding $\iota$ which is given, for $T\in\mathcal E'(\R)$, by $\iota(T)=[T\conv\phi_\eps]$, where the mollifier $\phi\in\ster{\Cnt[\infty]}(\R)$ satisfies $|\widehat\phi(\xi)|\le C_p \langle\xi\rangle^{-p}$ for each $p\in\N$ (for some $C_p\in\R$; in particular, this includes the case where $\phi\in\mathcal S(\R)$; the proof can easily be extended if the inequality only holds for $C_p=|\ln\rho|=[|\ln\eps|]\in\ster\R$, thus including also the case where $\phi$ is given by a mollifier as in \cite[Lemma 3.1]{OT98}). As usual, $\phi_\eps(x):= \eps^{-1}\phi(x/\eps)$.

\begin{ex}
Since $\mathcal U$ is an ultrafilter, $S\in \mathcal U$ or $\co S\in\mathcal U$ for any $S\subseteq(0,1]$. Let $\eps_n:= 1/2^{(n^n)}$. This sequence has the property that $(\forall p\in\N)$ $(\exists N)$ $(\forall n\ge N)$ $(\eps_{n+1}\le \eps_n^p)$. 
Consider $S:=\bigcup_{n\in\N} (\eps_{6n+3}, \eps_{6n}]$. We consider first the case that $S\in\mathcal U$. Then let
\begin{equation}\label{eq-counter}
\widehat T(\xi) := \sum_{m=1}^\infty \psi(\xi-\eps_{6m-1}^{-1}),
\end{equation}
where $\psi\in\mathcal S(\R)$ with $\widehat\psi\in\test(\R)$.

We claim that $T\in{\mathcal E}'(\R)\setminus\Cnt[\infty](\R)$, but $\iota(T)=[T\conv\phi_\eps]\in{}^\rho\mathcal E^\infty(\R)$.
\begin{proof}
With the usual estimates (e.g.\ Peetre's inequality), one sees that the sum \eqref{eq-counter} converges uniformly on compact subsets and that $\widehat T$ is a bounded $\Cnt[\infty]$-function. The sum thus converges in $\mathcal S'(\R)$, and $T\in\mathcal S'(\R)$ is well-defined. As
\[\langle T,\varphi\rangle = \sum_{m=1}^\infty \int_\R{\mathcal F}^{-1}(\psi)(x)e^{i\eps_{6m-1}^{-1}x}\varphi(x)\,dx,\quad\forall\varphi\in\mathcal S(\R),\]
$\supp T\subseteq \supp({\mathcal F}^{-1}(\psi))$, hence $T\in\mathcal E'(\R)$. Seeking a contradiction, suppose that $T\in\Cnt[\infty](\R)$. Then $T\in \mathcal S(\R)$, hence $\widehat T\in\mathcal S(\R)$, too, but due to its definition, $\widehat T(\xi)\not\to 0$ as $\xi\to+\infty$. It remains to be shown that $\iota(T)\in{}^\rho\mathcal E^\infty(\R)$. We first show that $\widehat{\iota(T)}$ satisfies the analogue of equation \eqref{ss-sup}. Let $p\in\N$ arbitrary ($p>0$) and $\eps\in (\eps_{6n+3}, \eps_{6n}]$. We proceed to show that for $n$ sufficiently large,
\begin{equation}\label{eq-counter2}
\langle \xi\rangle^p |\widehat{T\conv\phi_\eps}(\xi)| = \langle \xi\rangle^p|\widehat T(\xi)\widehat\phi(\eps\xi)| \le \sum_{m=1}^{\infty} \langle \xi\rangle^p|\psi(\xi-\eps_{6m-1}^{-1})\widehat\phi(\eps\xi)|\le \eps^{-1}.
\end{equation}
As $\psi\in\mathcal S(\R)$, we find $C_p\in\R$ such that (by Peetre's inequality)
\[\sum_{m=1}^{n} \langle \xi\rangle^p|\psi(\xi-\eps_{6m-1}^{-1})|
\le C_p\sum_{m=1}^{n}\langle \xi\rangle^p\langle\xi-\eps_{6m-1}^{-1}\rangle^{-p}
\le C'_p \sum_{m=1}^{n}\eps_{6m-1}^{-p}\le C''_p\eps_{6n-1}^{-p},\]
which is at most $\eps_{6n}^{-1}\le \eps^{-1}/2$, as soon as $n$ is sufficiently large.\\ Further, if $|\xi|\le \eps_{6n+4}^{-1}$ and $m>n$, then 
\[|\psi(\xi-\eps_{6m-1}^{-1})|\le C\langle\xi-\eps_{6m-1}^{-1}\rangle^{-1}\le C\langle\eps_{6m-1}^{-1}-\eps_{6n+4}^{-1}\rangle^{-1}\le 2C\eps_{6m-1}\]
and thus
\begin{align*}
\sum_{m=n+1}^\infty \langle \xi\rangle^p|\psi(\xi-\eps_{6m-1}^{-1})\widehat\phi(\eps\xi)|
&\le C_p\eps_{6n+4}^{-p}\sum_{m=n+1}^\infty \eps_{6m-1}\le 2C_p\eps_{6n+4}^{-p}\eps_{6n+5}\le 1
\end{align*}
as soon as $n$ is sufficiently large. On the other hand, if $|\xi|\ge\eps_{6n+4}^{-1}$, then
\[|\widehat\phi(\eps\xi)|\le C_p \langle\eps\xi\rangle^{-p-1}\le C'_p \eps^{-p-1}\langle\xi\rangle^{-p-1}\le C''_p\eps_{6n+3}^{-p-1}\langle\xi\rangle^{-p} \eps_{6n+4}\le \langle\xi\rangle^{-p}\]
as soon as $n$ is sufficiently large, and thus
\begin{align*}
\sum_{m=n+1}^\infty \langle \xi\rangle^p|\psi(\xi-\eps_{6m-1}^{-1})\widehat\phi(\eps\xi)|
&\le \sum_{m=n+1}^\infty |\psi(\xi-\eps_{6m-1}^{-1})|\le C.
\end{align*}
This proves \eqref{eq-counter2}, which implies, still for $\eps\in(\eps_{6n+3}, \eps_{6n}]$ and $n$ sufficiently large,
\[|D^k (T\conv\phi_\eps)(x)| = |\mathcal F^{-1}(\xi^k \cdot \widehat{T\conv\phi_\eps})(x)| \le \int_\R \langle\xi\rangle^k |\widehat{T\conv\phi_\eps}(\xi)|\,d\xi\le \eps^{-1}\int_\R \frac{d\xi}{\langle\xi\rangle^2},\]
hence $\iota(T)\in {}^\rho\mathcal E^\infty(\R)$.
\end{proof}
In the case when $\co S=\bigcup_{n\in\N}(\eps_{6n},\eps_{6n-3}]\in\mathcal U$, we similarly consider $\widehat T(\xi):=\sum_{m=0}^\infty \psi(\xi-\eps_{6m+2}^{-1})$.
\end{ex}

\end{document}